\documentclass[a4paper,11pt,titlepage,twoside]{article}

\usepackage{graphicx}
\usepackage[T1]{fontenc} 
\usepackage[utf8]{inputenc}
\usepackage[english]{babel}
\usepackage{soul}
\usepackage{amsfonts}
\usepackage{amsmath}
\usepackage{amsthm}
\usepackage{amssymb}
\usepackage{mathrsfs}
\usepackage[top=5cm, bottom=3cm, left=3.5cm, right=3.5cm]{geometry}
\usepackage{setspace}
\usepackage{afterpage}
\usepackage{extarrows}
\usepackage{fancyhdr}
\usepackage{titlesec}
\usepackage{enumitem} \setlist{nosep}

\theoremstyle{definition}
\newtheorem{defin}{Definition}[section]
\theoremstyle{plain}
\newtheorem{theo}[defin]{Theorem}

\newtheorem{cor}[defin]{Corollary}
\theoremstyle{definition}

\newtheorem{rem}[defin]{Remark}
\newtheorem{rems}[defin]{Remarks}

\renewcommand{\H}{\mathcal{H}}

\newcommand{\B}{\mathcal{B}}
\newcommand{\G}{\mathcal{G}}

\newcommand{\n}[1]{\|#1\|}
\newcommand{\nor}{\|\cdot\|}

\renewcommand{\l}{\langle}
\renewcommand{\r}{\rangle}
\newcommand{\N}{\mathbb{N}}
\newcommand{\Z}{\mathbb{Z}}
\newcommand{\R}{\mathbb{R}}
\newcommand{\Q}{\mathbb{Q}}
\newcommand{\C}{\mathbb{C}}

\newcommand{\pint}{\l\cdot,\cdot\r}
\newcommand{\pin}[2]{\l#1 , #2\r}
\newcommand{\no}{\noindent}
\newcommand{\ol}{\overline}

\newcommand{\mez}{\frac{1}{2}}

\numberwithin{equation}{section}

\renewcommand\labelenumi{\emph{(\roman{enumi})}}
\renewcommand\theenumi\labelenumi

\titleformat{\section}
{\normalfont\fillast \fontsize{12}{15}\scshape}{\thesection.}{0.8em}{}

\titleformat{\subsection}
{\normalfont\fillast \fontsize{11}{12}\scshape}{\thesubsection.}{0.8em}{}

\begin{document}
	
\thispagestyle{plain}

\begin{center}
	\large
	{\uppercase{\bf Orbits of bounded bijective operators \\and Gabor frames}}\\
	\vspace*{0.4cm}
	{\scshape{Rosario Corso}}
\end{center}

\normalsize 
\vspace*{0.2cm}	

\small 

\begin{minipage}{11.8cm}
	{\scshape Abstract.} 
	This paper is a contribution to frame theory. Frames in a Hilbert space are generalizations of orthonormal bases. In particular, Gabor frames of $L^2(\mathbb{R})$, which are made of translations and modulations of one or more windows, are often used in applications. 
	
	More precisely, the paper deals with a question posed in the last years by Christensen and Hasannasab about the existence of overcomplete Gabor frames, with some ordering over $\Z$, which are orbits of bounded operators on $L^2(\R)$. 
	Two classes of overcomplete Gabor frames which cannot be ordered over $\Z$ and represented by orbits of operators in $GL(L^2(\R))$ are given. 
	
	Some results about operator representation are stated in a general context for arbitrary frames, covering also certain wavelet frames. 
\end{minipage}

\vspace*{.5cm}

\begin{minipage}{11.8cm}
\end{minipage}

\vspace*{.5cm}

\begin{minipage}{11.8cm}
\end{minipage}

\normalsize

\section{Introduction}

Let $\H$ be a separable Hilbert space with inner product $\pint$ and norm $\nor$. A sequence $\{f_n\}_{n\in I}\subset \H$, indexed by a countable set $I$, is called a {\it frame} of $\H$ if there exist $0<A,B$ such that 
$$
A\n{f}^2\leq \sum_{n\in I} |\pin{f}{f_n}|^2\leq B \n{f}^2, \qquad \forall f\in \H.
$$
Given a frame $\{f_n\}_{n\in I}$ of $\H$ there exists (at least) a frame $\{h_n\}_{n\in I}$ of $\H$ such that the following formulas hold
$$f=\sum_{n\in I} \pin{f}{h_n}f_n=\sum_{n\in I} \pin{f}{f_n}h_n \qquad \forall f\in \H.$$ 
Frames are then generalizations of orthonormal bases and they find many engineering applications (for an introduction about frames one could refer to \cite{Chris,Groc,Heil}). 

A new aspect in frame theory appeared recently is the {\it operator representation} of a frame $\{f_n\}_{n\in I}$ (where $I=\N_0:=\N \cup \{0\}$ or $I=\Z$) of $\H$, that is the possibility to write $\{f_n\}_{n\in I}$ as {\it orbit over $I$} 
\begin{equation}
\label{repr_intro}
\{T^nf_0\}_{n\in I},  
\end{equation}
of some operator $T$ acting on $\H$. This aspect has been intensively studied by Christensen et al. \cite{Chris_o17_sampta,Chris_o17_Z,Chris_o18_open,Chris_o18_char,Chris_o18} and it is connected to the so-called {\it dynamical sampling} (see also \cite{ACCP,ACMP,Aldroubi_orig,Aldroubi_proc,Cabrelli_finite,Cabrelli_sampta,CMS,Philipp}). A variation of the problem was also given in \cite{Chris_appr,Chris_o19_sampta}.

First of all we need to precise the meaning of \eqref{repr_intro} for $I=\Z$: $T$ is an operator whose restriction to span$(\{f_n\}_{n\in \Z})$ is bijective and $T^n f_0$ is the iteration of this restriction. Clearly, there is no ambiguity if $T$ belongs in  $GL(\H)\subset \B(\H)$, where $\B(\H)$ is the set of bounded operators $T:\H\to \H$ and $GL(\H)$ is the subset of operators $T\in \B(\H)$ such that $T^{-1}$ exists and $T^{-1}\in \B(\H)$. We note that the expression \eqref{repr_intro} for $I=\Z$ and $T\in GL(\H)$ is quite familiar in frame theory. Indeed, shift-invariant sequences $\{T_{na} \phi\}_{n\in \Z}:=\{\phi(\cdot -na)\}_{n\in \Z}$ with $\phi\in L^2(\R)$ have this expression (and the translation operator $T_a$ is even unitary).

In general, the existence of a representation \eqref{repr_intro} is equivalent to the linear independence of $\{f_n\}_{n\in I}$. Nevertheless, an operator $T$ as in \eqref{repr_intro} is not necessarily bounded. 
The following necessary and sufficient condition for the existence  of an operator $T\in \B(\H)$, in the case $I=\N_0$, was found in \cite{Chris_o18}: the kernel of the synthesis operator of $\{f_n\}_{n\in \N_0}$ is invariant under right-shift. For this reason, a Riesz basis $\{f_n\}_{n\in \N_0}$ is always represented by a bounded operator, {\it whatever} the ordering of $\{f_n\}_{n\in \N_0}$ is. On the other hand, a linearly independent frame $\{f_n\}_{n\in \N_0}$ with finite positive excess is not represented by a bounded operator {\it whatever} the ordering of $\{f_n\}_{n\in \N_0}$ is. \\
The case $I=\Z$ is analogous \cite{Chris_o17_Z}: a linearly independent frame is generated by an operator in $GL(\H)$ if and only if the kernel of the synthesis operator is both left and right shift invariant. The similar conclusions about Riesz bases and frames with finite excess hold. 
When we assume that the operator is just bounded (or it has a bounded inverse), then the equivalence is valid with shift only on the right (or on the left).   

A class of frames widely used in applications is constituted by Gabor frames, see \cite{Chris,FeiStr,Groc,Heil}. In this paper, we consider Gabor frames for the space $L^2(\R)$ having the form
$$\G(g,a,b)=\{e^{2\pi i m b x}g(x-na)\}_{m,n\in \Z},$$
with some window $g\in L^2(\R)$ and $a,b>0$. 

Motivated by the fact that a Gabor frame 
is linearly independent and it has infinite excess if overcomplete, the following question was posed by Christensen and Hasannasab  \cite{Chris_o17_sampta}: does there exists an ordering over $I$ of an overcomplete Gabor frame such that \eqref{repr_intro} holds with an operator $T\in\B(L^2(\R))$? Note that $\G(g,a,b)$ is actually made of iterations of two unitary operators (translation and modulation).  
Coming back to the question, a complete negative answer was found in \cite{Chris_o18_char} for the case $I=\N_0$. More precisely, any norm-bounded overcomplete frame with any indexing is not an orbit over $\N_0$ of an operator in $\B(\H)$.

In this paper we deal with the open question for $I=\Z$ under the additional condition $T\in GL(L^2(\R))$. We show that if $\G(g,a,b)$ is an overcomplete Gabor frame, then there is no ordering $\{f_n\}_{n\in \Z}$ of $\G(g,a,b)$ such that \eqref{repr_intro} holds with $T\in GL(L^2(\R))$ if at least one of the following conditions is satisfied:
\begin{enumerate}
	\item[(i)] the window $g$ has support in a finite interval of length $a$ (Corollary  \ref{cor_Gabor_1});
	\item[(ii)] $ab\in \R\backslash\Q$ (Theorem \ref{th_main2}).
\end{enumerate}
Roughly speaking, the motivation is that `there are too many elements in $\G(g,a,b)$'. These results make use of Theorem \ref{th_char_GL_Z} about the following new characterization:
a frame $\{f_n\}_{n\in \Z}$ of $\H$ is an orbit of an operator in $GL(\H)$ if and only if $\pin{f_i}{h_j}=\pin{f_{i+1}}{h_{j+1}}$ for all $i,j\in \Z$, where $\{h_n\}_{n\in \Z}$ is the canonical dual of $\{f_n\}_{n\in \Z}$. When $\{f_n\}_{n\in \Z}$ is tight this condition is simplified because, as known,  $\{f_n\}_{n\in \Z}$ coincides (up to a constant) with its  canonical dual, and what is more is that if a tight frame is generated by an operator $T$ in $GL(\H)$ over $\Z$, then $T$ is necessarily unitary. 

More generally, a frame whose canonical tight frame is an infinite union of frame sequences of orthogonal pairwise subspaces is not generated over $\Z$ by an operator in $GL(\H)$ (Theorem \ref{th_main1}).  This includes also particular  wavelet frames.

We further prove that if a frame is an orbit over $\Z$ of a normal operator $T$ in $GL(\H)$, then $T$ is unitary but not self-adjoint (Corollary \ref{cor_normal}).

\vspace*{-0.1cm}

\section{Preliminaries}
\label{sec:prel}

Let $\H$ be a separable Hilbert space with inner product $\pint$ and norm $\nor$. The set of bounded linear operator everywhere defined on $\H$ is denoted by $\B(\H)$. The symbol $GL(\H)$ stands for the set of $T\in \B(\H)$ such that $T:\H\to \H$ is a bijection.

Let $I$ be one of the index set $\N_0=\N \cup \{0\}$ and $\Z$. We denote by $l_2(I)$ the classical Hilbert space of complex sequences indexed by $I$. We define by $\mathcal{T}$ the {\it right-shift operator} on $l_2(\Z)$, $\mathcal{T}\{c_n\}_{n\in \Z}=\{c_{n+1}\}_{n\in \Z}$
for $\{c_n\}_{n\in \Z}\in l_2(\Z)$. 
A sequence $\{f_n\}_{n\in I}\subset \H$ is called a {\it frame} of $\H$ if there exist $0<A,B$ such that 
$$
A\n{f}^2\leq \sum_{n\in I} |\pin{f}{f_n}|^2\leq B \n{f}^2, \qquad \forall f\in \H,
$$
and $A,B$ are called {\it lower} and {\it  upper bounds}, respectively. When we write $\{f_n\}_{n\in I}$ we mean that a precise ordering  (or, with an equivalent word, indexing) over $I$ is considered. Sometimes, we will indicate a frame with one letter only, e.g. $\mathcal{F}$, when we want to stress that a property does not depend on a precise ordering. 
For many aspects the indexing is irrelevant (for example the nature of frame of $\{f_n\}_{n\in I}$) and for others aspects it is important (for instance the operator representation). 

Let $\{f_n\}_{n\in I}$ be a frame of $\H$. The {\it synthesis operator} $U:l_2(I)\to \H$ of $\{f_n\}_{n\in I}$ is defined as 
$$
U\{c_n\}_{n\in \Z}:=\sum_{n\in I} c_n f_n, \qquad \forall \{c_n\}_{n\in \Z}\in l_2
$$
and it is bounded, surjective with kernel $N(U)=\{\{c_n\}\in l_2(I):\sum_{n\in I} c_n f_n=0\}$. 
The {\it frame-operator} $S:\H\to \H$ of $\{f_n\}_{n\in I}$ is defined as 
$$
Sf:=\sum_{n\in I} \pin{f}{f_n} f_n, \qquad \forall f\in \H,
$$
and it is a self-adjoint positive operator. More precisely, $S=UU^*$ and $S\in GL(\H)$. A frame with $N(U)=\{0\}$ is called a {\it Riesz basis}. As it is well-known, the notion of Riesz basis has several equivalent formulations (see \cite[Theorem 7.1.1]{Chris}).  A frame which is not a Riesz basis is called {\it overcomplete}. The {\it excess} of a frame is the number of elements that can be removed yet
leaving a frame. Moreover, the excess of a frame coincides with the dimension of $N(U)$,  see \cite[Lemma 4.1]{BCHL_excess}. \\
Two classical frames associated to $\{f_n\}_{n\in I}$ are the {\it canonical dual frame} $\{h_n\}_{n\in I}:=\{S^{-1}f_n\}_{n\in I}$ and the {\it canonical tight frame} $\{t_n\}_{n\in I}:=\{S^{-\mez}f_n\}_{n\in I}$. 
The importance of the canonical dual is the reconstruction formula $$f=\sum_{n\in I} \pin{f}{h_n}f_n=\sum_{n\in I} \pin{f}{f_n}h_n \qquad \forall f\in \H.$$ 
In particular, span($\{f_n\}_{n\in I}$) is dense in $\H$. 
Clearly, one of $\{f_n\}_{n\in I}$, $\{h_n\}_{n\in I}$ and $\{t_n\}_{n\in I}$ is linearly independent if and only if anyone else is linearly independent. When $\{f_n\}_{n\in I}$ is tight with bound $A$ then $\{t_n\}_{n\in I}=\{\frac{1}{\sqrt{A}}f_n\}_{n\in I}=\{\sqrt{A}h_n\}_{n\in I}$. 

If $\{f_n\}_{n\in I}$ is a frame of $\H$ and $\{f_n\}_{n\in I}=\{T^n f_0\}_{n\in I}$ for some operator $T$ on $\H$ with domain containing span$\{f_n\}_{n\in I}$, then we say that $\{f_n\}_{n\in I}$ is an {\it orbit over $I$ of} $T$ (or, $\{f_n\}_{n\in I}$ {\it is generated by $T$ over $I$}). The operator representation of frames has been studied mainly in \cite{Chris_o17_acha,Chris_o18_char,Chris_o18} for $I=\N$ and in \cite{Chris_o17_Z,Chris_o18_char} for $I=\Z$. 
Almost always in this paper $I$ will be $\Z$. 

We summarize in the following theorem some results contained in 
\cite[Proposition 2.1]{Chris_o17_Z}, \cite[Corollary 4.5]{Chris_o18_char} about the operator representation of frames indexed by $\Z$.

\begin{theo}
	\label{th_nor_T}
	Let $\mathcal{F}:=\{f_n\}_{n\in \Z}$ be a linearly independent frame with lower bound $A$ and upper bound $B$. 
	\begin{enumerate}
		\item There exists a bijective linear operator $T_\mathcal{F}:span\{\mathcal{F}\}\to span\{\mathcal{F}\}$ such that $\{f_n\}_{n\in \Z}=\{T_\mathcal{F}^nf_0\}_{n\in \Z}$. 
		\item There exists an operator $T\in GL(\H)$ such that $\{f_n\}_{n\in \Z}=\{T^nf_0\}_{n\in \Z}$ if and only if $\mathcal{T} N(U)\subseteq N(U)$ and $\mathcal{T}^{-1} N(U)\subseteq N(U)$. 		
		Moreover, in this case $T$ is uniquely determined, it is the closure of $T_\mathcal{F}$ and for all $n\in \Z$ 
		\begin{align}
		\label{eq_norm_bound}
		\sqrt{\frac{A}{B}}\n{f} \leq \n{T^n f}\leq \sqrt{\frac{B}{A}}\n{f}, \\
		\sqrt{\frac{A}{B}}\n{f} \leq \n{(T^*)^n f}\leq \sqrt{\frac{B}{A}}\n{f}, \nonumber
		\end{align}
		i.e., $1 \leq \n{T}\leq \sqrt{\frac{B}{A}}$ and $1 \leq \n{T^{-1}}\leq \sqrt{\frac{B}{A}}$.
	\end{enumerate}
\end{theo}

\begin{rems}
	\label{rem_th_lect}
\begin{enumerate}
	\item[(i)] The assumption that $\{f_n\}_{n\in \Z}$ is linearly independent is not restrictive. Indeed, this is the only case when $\{f_n\}_{n\in \Z}=\{V^nf_0\}_{n\in \Z}$ for some operator $V$ (see   
	\cite[Proposition 2.1]{Chris_o17_Z}). Regular Gabor frames, that we will consider in the next section, are always linearly independent. The same property is enjoyed also by other frames, like some classes of wavelet frames \cite{Chris}. 
	\item[(ii)] If $T$ is an injective operator in $\B(\H)$ such that $\{f_n\}_{n\in \Z}=\{T^nf_0\}_{n\in \Z}$ is a frame of $\H$ then $T$ belongs necessarily in $GL(\H)$. Indeed, it is surjective because for every $f\in \H$ there exists $\{c_n\}\in l_2(\Z)$ such that $f=\sum_{n\in \Z}c_n f_n=T\sum_{n\in \Z}c_n f_{n-1}$. 
	\item[(iii)] Given a frame $\mathcal{F}:=\{f_n\}_{n\in \Z}$, if there exists $T\in GL(\H)$ such that $\{f_n\}_{n\in \Z}=\{T^nf_0\}_{n\in \Z}$, then $T$ is uniquely determined and it is the closure of $T_\mathcal{F}$.
	\item[(iv)] A well-known consequence of this theorem is that a  Riesz basis (ordered in any way over $\Z$) is the orbit of an operator in $GL(\H)$. 
	Moreover, if an overcomplete linearly independent frame is an orbit of an operator in $GL(\H)$, then it has infinite excess (see \cite[Corollary 2.5]{Chris_o17_Z}). 
	\item[(v)] If a frame of $\H$ is an orbit over $\Z$ of an operator in $GL(\H)$, then it is norm-bounded below. 
	\item[(vi)] From \eqref{eq_norm_bound} follows, more generally, that $1 \leq \n{T^n}\leq \sqrt{\frac{B}{A}}$ for any $n\in \Z$. 
	\item[(vii)] A shift-invariant system $\{\phi(\cdot-na)\}_{n\in\Z}$ where $\phi \in L^2(\R)$, has a natural representation as orbit over $\Z$ of a unitary operator (the translation operator). As known,  $\{\phi(\cdot-na)\}_{n\in\Z}$ is never a frame, but can be a frame sequence (i. e. a frame of its closed span) according to the function $\Phi:\R\to \R$, $\Phi(\gamma)=\sum_{k\in \Z} |\widehat{\phi}(\frac{\gamma+k}{b})|^2$ (see \cite{Chris}). 
	The following example can be seen as a particular case of shift-invariant system. Let  $g\in L^2(0,1)$ such that for some $0<m,M$ one has 
	$m\leq |g(x)|\leq M$ for almost $x\in (0,1)$. Moreover, let $b=\frac{1}{N}$ with $N\in \N$. Thus $\mathcal{E}(g,b):=\{g_m\}_{m\in \Z}$,  with  $g_m(x)=e^{2\pi i m b x} g(x)$, is a frame of $L^2(0,1)$ and $\mathcal{E}(g,b)=\{E_1^m g\}_{m\in \Z}$, where $(E_1 f)(x)=e^{2\pi i b x}f(x)$ for all $f\in L^2(0,1)$. 
	
\end{enumerate}
\end{rems}

It is worth recalling also the following classification found in \cite{Chris_o18_char}: given a frame $\{f_n\}_{n\in \Z}=\{T^nf_0\}_{n\in \Z}$ of $\H$, where $T\in GL(\H)$, there exist a Borel set $\sigma$ of the unit circle $\mathbb{T}$ and a bounded bijection $V:L^2(\sigma)\to \H$ such that $f_n=VM_\sigma^n 1_\sigma$ for all $n\in \Z$ (here $M_\sigma:L^2(\sigma)\to L^2(\sigma)$ is the operator $(M_\sigma h)(z)=z h(z)$ and $1_\sigma$ is the function on $\sigma$ constant to 1). Moreover, $\{f_n\}_{n\in \Z}$ is a Riesz basis if and only if $\sigma=\mathbb{T}$.

We now show other direct consequences of Theorem \ref{th_nor_T}.

\begin{cor}
	\label{cor_tight}
	If $\{f_n\}_{n\in \Z}=\{T^nf_0\}_{n\in \Z}$, with $T\in GL(\H)$, is a tight frame, then $T$ is unitary.
\end{cor}
\begin{proof}
	By Theorem \ref{th_nor_T}, $T$ is an isometry. Since  $T\in GL(\H)$, $T$ is also unitary. 
\end{proof}

Taking into account that for a normal operator $R\in GL(\H)$ the inequalities $m\leq\n{R}\leq M$ imply $\frac{1}{M}\leq\n{R^{-1}}\leq \frac{1}{m}$ we prove the following result.

\begin{cor}
	\label{cor_normal}
	Let $\{f_n\}_{n\in \Z}$ be a frame such that $\{f_n\}_{n\in \Z}=\{T^nf_0\}_{n\in \Z}$ for some $T\in GL(\H)$. The following statements hold.
	\begin{enumerate}
		\item If $T^s$ is normal for some $s\in \Z$, then $T^s$ is unitary.
		\item A power $(T^*)^s$ of $T^*$ with $s\in \Z$ cannot be equal to a polynomial $p(T)$ on $T$. In particular, $T$ cannot be self-adjoint. 
	\end{enumerate}
\end{cor}
\begin{proof}
	Suppose that $T^s$ is normal for some $s\in \Z$. Then also $T^{-s}$ is normal and thus $\n{T^s}$ is necessarily equals to $1$ by Remark \ref{rem_th_lect}(vi) and the observation above. This means that $T^s$ is unitary. 
	To prove (ii), suppose that  $(T^*)^s=p(T)$ for some $s\in \Z$ and some polynomial $p$. Clearly, $T^s$ is normal and then unitary for the first statement. We conclude that $f_{-s}=T^{-s}f_0=(T^*)^sf_0=p(T)f_0\in \text{span}({\{f_n\}}_{n\in \Z})$, i.e. a contradiction to the linear independence of  $\{f_n\}_{n\in \Z}$. 
\end{proof}

Frames are not orbits over $\Z$ of self-adjoint operators in $GL(\H)$, despite the existence of frames which are orbits over $\N_0$ of bounded self-adjoint operators (see \cite[Example 4]{Aldroubi_normal}). 
Unitary operators instead appear as generators over $\Z$ of some class of frames. Thus it useful to have some characterizations of such frames like the following result that involves inner products between elements.

\begin{theo}
	\label{th_unitary}
	A linearly independent frame $\mathcal{F}:=\{f_n\}_{n\in \Z}$ is generated by an unitary operator over $\Z$ if and only if $\pin{f_i}{f_j}=\pin{f_{i+1}}{f_{j+1}}$ for all $i,j\in \Z$.
\end{theo}
\begin{proof}
	One implication is trivial. For the other one, let $T_\mathcal{F}$ be the operator as in Theorem \ref{th_unitary} and suppose that $\pin{f_i}{f_j}=\pin{f_{i+1}}{f_{j+1}}$ for all $i,j\in \Z$. For $f\in \;$span$(\mathcal{F})$, $f=\sum_{i} c_i f_i$ (finite sum), we have
	\begin{align*}
	\n{T_\mathcal{F} f}^2&=\pin{\sum_{i} c_i f_{i+1}}{\sum_{i} c_i f_{i+1}}=\sum_{i,j} c_i\ol{c_j} \pin{f_{i+1}}{f_{j+1}}\\
	&=	\sum_{i,j} c_i\ol{c_j} \pin{f_{i}}{f_{j}}=\pin{\sum_{i} c_i f_{i}}{\sum_{i} c_i f_{i}}=\n{f}^2.
	\end{align*}
	In an analogous way, $\n{T_\mathcal{F}^{-1} f}=\n{f}$ for all $f\in \;$span$(\mathcal{F})$. Thus $T_\mathcal{F}$ extends to a unitary operator $T$ on $\H$ and $f_n=T^n f_0$ for all $n\in \Z$.
\end{proof}

Note that this theorem is an improvement of Proposition 3.5 of \cite{Chris_o17_Z}. Indeed we do not assume a priori that $\mathcal{F}$ is generated by a bounded operator in Theorem \ref{th_unitary}. Collecting the results above, we now state another characterization of frames which are orbits of operators in $GL(\H)$ over $\Z$.

\begin{theo}
	\label{th_char_GL_Z}
	Let $\{f_n\}_{n\in \Z},\{h_n\}_{n\in \Z}$ and $\{t_n\}_{n\in \Z}$ be a linearly independent frame, its canonical dual and its canonical tight frame. The following statements are equivalent.
	\begin{enumerate}
		\item $\{f_n\}_{n\in \Z}$ is generated by an operator in $GL(\H)$ over $\Z$.
		\item $\{t_n\}_{n\in \Z}$ is generated by an operator in $GL(\H)$ over $\Z$.
		\item $\{t_n\}_{n\in \Z}$ is generated by an unitary operator over $\Z$.
		\item $\pin{t_i}{t_j}=\pin{t_{i+1}}{t_{j+1}}$ for all $i,j\in \Z$.
		\item $\pin{f_i}{h_j}=\pin{f_{i+1}}{h_{j+1}}$ for all $i,j\in \Z$.
	\end{enumerate}
\end{theo}
\begin{proof}
	(i)$\Longleftrightarrow$(ii)  An operator $T$ generates $\{f_n\}_{n\in \Z}$ if and only if $S^{-\mez}TS^{\mez}$ is generates $\{t_n\}_{n\in \Z}$. 
	Therefore, the statement follows by the fact that $T\in GL(\H)$ if and only if $S^{-\mez}TS^{\mez}\in GL(\H)$. \\
	(ii)$\Longleftrightarrow$(iii) The non trivial implication is proved in Corollary \ref{cor_tight}.\\
	(iii)$\Longleftrightarrow$(iv) It is given by Theorem \ref{th_unitary}.\\
	(iv)$\Longleftrightarrow$(v) The statement is clear taking into account that $\pin{t_i}{t_j}=\pin{f_i}{h_j}$ for all $i,j\in \Z$.
\end{proof}

\begin{rems}
	\label{rem_th_char_GL_Z}
	\begin{enumerate}
		\item[(i)] Theorem \ref{th_char_GL_Z} gives another proof of the fact that a Riesz basis $\{f_n\}_{n\in \Z}$ is generated by an operator in $GL(\H)$. 
		Indeed,	$\pin{f_i}{h_j}=1$ if $i=j$ and $\pin{f_i}{h_j}=0$ if $i\neq j$.
		\item[(ii)] We can give also a simple proof that the excess of an overcomplete frame $\{f_n\}_{n\in \Z}$ generated by an operator in $GL(\H)$ is infinite (see Remark \ref{rem_th_lect}(iv)). Indeed, by \cite[Proposition 5.5]{BCHL_excess} the excess of $\{f_n\}_{n\in \Z}$ is equal to $\sum_{n\in \Z} (1-\pin{f_n}{h_n})$ where $\{h_n\}_{n\in \Z}$ is the canonical dual. Hence, Theorem \ref{th_char_GL_Z} says that $\sum_{n\in \Z} (1-\pin{f_n}{h_n})=\infty$.
		\item[(iii)] Point (v) of Theorem \ref{th_char_GL_Z} can be formulate by saying that the cross Gram matrix $(a_{i,j})_{i,j\in \Z}=(\pin{f_i}{h_j})_{i,j\in \Z}$ is a Toeplitz matrix. When $\{f_n\}_{n\in \Z}$ is a tight frame, $(a_{i,j})_{i,j\in \Z}$ is a multiple of the Gram matrix of $\{f_n\}_{n\in \Z}$.
		\item[(iv)] We mention that if $\{f_n\}_{n\in \Z}$ is generated by $T\in GL(\H)$ then $\{h_n\}_{n\in \Z}$ is generated by $(T^*)^{-1}$ (see \cite[Proposition 3.2]{Chris_o17_Z}). This remark can be used to prove  the implication (i)$\Rightarrow$(v) in a different way.
		\item[(v)] In Theorem \ref{th_char_GL_Z} the frame is indexed by $\Z$. In the case $I=\N_0$, the first three statements change as follows. Let $\{f_n\}_{n\in \N_0}$ and $\{t_n\}_{n\in \N_0}$ be a linearly independent frame and its canonical tight frame. Then,  $\{f_n\}_{n\in \N_0}$ is generated by a bounded operator if and only if  $\{t_n\}_{n\in \N_0}$ is generated by a bounded operator if and only if $\{t_n\}_{n\in \N_0}$ is generated by an operator with norm $1$. However, it can be proved that $\{t_n\}_{n\in \N_0}$ is generated by an isometric operator if and only if it is a multiple of an orthonormal basis. Indeed, \cite[Corollary 3.6]{Chris_o18_char} ensure that that if $\{f_n\}_{n\in \N_0}=\{T^nf_0\}_{n\in \N_0}$ is an overcomplete frame with $T\in \B(\H)$, then $\n{f_n}\to 0$. 
		\end{enumerate}
\end{rems}

Now we state some necessary conditions for a frame to have a representation with an operator in $GL(\H)$ as consequence of Theorem \ref{th_char_GL_Z}. In particular, the last assertion follows by \cite[Exercise 3.6]{Chris}.

\begin{cor}
	\label{cor_nec}
	Let $\mathcal{F}$ be a linearly independent frame with canonical dual $\mathcal{D}$ and frame-operator $S$. Suppose that $\mathcal{F}$ has an ordering over $I=\Z$ such that it is generated by an operator in $GL(\H)$. Then 
	\begin{enumerate}
		\item the set $Ip(f):=\{\pin{f}{h}:h\in \mathcal{D}\}$ is invariant with respect to $f\in \mathcal{F}$;
		\item for every $c\in \C$ the cardinality $\mathsf{m}_c(f)=|\{h\in \mathcal{D}: \pin{f}{h}=c\}|$ is invariant with respect to $f\in \mathcal{F}$;
		\item the number $\pin{f}{S^{-1} f}$ is constant for all $f\in \mathcal{F}$.
	\end{enumerate}	
	Suppose that $\mathcal{F}$ is tight with bounds $A=B$. Then 
	\begin{enumerate}
		\item[\emph{(iv)}] $\mathcal{F}$ is norm-constant, i.e., $\n{f}=\n{f'}$ for all $f,f'\in \mathcal{F}$;
		\item[\emph{(v)}] if there is $f\in \mathcal{F}$ such that $\n{f}=\sqrt{B}$, then $\mathcal{F}$ is a multiple of an orthonormal basis.
	\end{enumerate}
\end{cor}

A special tight frame can be obtained by the union of $N$ orthonormal bases, for arbitrary $N\in \N$. If the obtained frame is linearly independent then Theorem \ref{th_char_GL_Z} applies to it. The frame $\mathcal{E}(\chi_{(0,1)},\frac{1}{N})$ covers this case; indeed, it is a union of $N$ orthonormal bases $\{f_{mN+k}\}_{m\in \Z}$, $k=1,\dots, N$ and it is an orbit of a unitary operator. The precise indexing is important, indeed if we interchange the elements $g_i$ and $g_j$ where $i-j$ is not a multiple of $N$, then with the new indexing $\mathcal{E}(\chi_{(0,1)},\frac{1}{N})$ is no more generated by an operator in $GL(L^2(0,1))$ as consequence of Theorem \ref{th_char_GL_Z}. We may ask if any linearly independent tight frame which is a finite union of orthonormal bases can be always ordered in a way so that it is generated by an operator in $GL(\H)$ (i.e., unitary) over $\Z$. The answer is negative and we will motivate it in Remark \ref{rem_union_orth_2}.

\vspace*{-0.1cm}
\section{Main results}
\label{sec:Gabor}

For $a,b\in \R$ define the operators $T_a$ and $E_b$ on $f\in L^2(\R)$ by 
$$
(T_af)(x)=f(x-a) \;\;\text{ and }\;\; (E_b f)(x)=e^{2 \pi i b x}f(x) \qquad \forall x \in \R.
$$
The operators $T_a$ and $E_b$ are unitary and are called {\it translation by $a$} and {\it modulation by $b$}. The following commutation rule holds: $T_aE_bf=e^{-2\pi i ab}E_bT_a$ (called Weyl relation in \cite{ReedSimon}). The {\it Gabor sequence} generated by $g\in L^2(\R)$, $a,b>0$ is 
$$
\G(g,a,b)=\{g_{m,n}\}_{m,n\in \Z}:=\{E^m_b T^n_a g\}_{m,n\in \Z},
$$
i.e. a  sequence obtained by the iterations of two unitary operators. 
Explicitly, $g_{m,n} (x)=e^{2\pi i bnx }g(x-ma)$ for all $x\in \R$. When $\G(g,a,b)$ is a frame of $L^2(\R)$ we call it a {\it Gabor frame} and we necessarily have $ab\leq 1$. 
Gabor frames have been widely studied in literature and we refer to the books \cite{Chris,Groc,Heil} for an introduction to this topic and, in particular, for the classical results we will use in this section.  

We will always consider $\G(g,a,b)$ a frame in the following. Recall that if  $\G(g,a,b)$ has frame-operator $S$, then its canonical dual frame is $\G(S^{-1}g,a,b)$ and its canonical tight frame is $\G(S^{-\frac{1}{2}}g,a,b)$.

Any Gabor frame $\mathcal{G}(g,a,b)$ is linearly independent therefore for a fixed indexing $\mathcal{G}(g,a,b)=\{f_n\}_{n\in I}$ where $I=\N_0$ or $I=\Z$ there exists an operator $T:\text{span}\{\mathcal{G}(g,a,b)\}\to \text{span}\{\mathcal{G}(g,a,b)\}$ such that $f_n=T^nf_0$ for any $n\in I$. 
In \cite{Chris_o18_char} it was proved that $\mathcal{G}(g,a,b)$ is an orbit over $I=\N$ of a bounded operator if and only if $\mathcal{G}(g,a,b)$ is a Riesz basis. As explained in the Introduction, here we deal with the following problem (\cite{Chris_o17_sampta,Chris_o18_open}).\\

{\bf Question:} Can an overcomplete Gabor frame $\mathcal{G}(g,a,b)$ be indexed over $\Z$ and represented by an operator in $GL(L^2(\R))$? \\

When $\mathcal{G}(g,a,b)$ is a Riesz basis then we know from Theorem \ref{th_nor_T} that it is trivially an orbit over $\Z$ of an operator in $GL(\H)$ for any ordering. As it is well-known, $\mathcal{G}(g,a,b)$ is a Riesz basis if and only if $ab=1$. Therefore, the problem of the representation of $\mathcal{G}(g,a,b)$ is more interesting case of overcompletness (i.e., if $ab<1$). The problem is even more interesting taking into account that, if $\mathcal{G}(g,a,b)$ is overcomplete, then it has infinite excess.

In the aim of finding a counterexample for some choice of $g, a$ and $b$ one could try to verify the necessary conditions of Corollary \ref{cor_nec}. Items (iii), (iv) and (v) do not lead to any conclusion since they are always satisfied by any Gabor frame $\mathcal{G}(g,a,b)$. 
Nevertheless, an analysis of condition (i) will give a negative answer in Theorem \ref{th_main2} assuming $ab\in \R\backslash\Q$. A negative answer in another particular case can be given as application of the following result.

\begin{theo}
	\label{th_main1}
	 	Let $\mathcal{F}$ be a linearly independent overcomplete frame of $\H$ with frame-operator $S$ such that 
	 \begin{enumerate}
	 	\item $\mathcal{F}=\cup_{k\in \Z} \mathcal{F}_k$ is an infinite disjoint union of non-empty subsets;
	 	\item $
	 	\pin{f}{S^{-1}f'}=0$ for all $f\in \mathcal{F}_{k_1}$ and 
	 	$f'\in \mathcal{F}_{k_2}$, with $k_1 \neq k_2.$
	 \end{enumerate}	
	 Then, regardless to the ordering over $\Z$,  $\mathcal{F}$ is not generated by an operator in $GL(\H)$.
\end{theo}

\no In other words, a linearly independent overcomplete frame whose canonical tight frame is an infinite union of frame sequences of orthogonal pairwise subspaces is not an orbit over $\Z$ of an operator in $GL(\H)$. 

\begin{proof}
	First of all, fix an ordering $\{f_n\}_{n\in \Z}$ of $\mathcal{F}$ and suppose that $\{f_n\}_{n\in \N}$ is generated by an operator in $GL(\H)$. Then the canonical tight frame $\{t_n\}_{n\in \Z}=\{S^{-\mez}f_n\}_{n\in \Z}$ is generated by a unitary operator $T$  and $\pin{t_i}{t_j}=\pin{t_{i+1}}{t_{j+1}}$ for all $i,j\in \Z$ by Theorem \ref{th_char_GL_Z}. \\
	The tight frame $\{t_n\}_{n\in \Z}$ is overcomplete, so we are able to find $k_0\in \Z$ such that $\mathcal{F}_{k_0}$ contains two distinct elements $t,t'$ whit $\pin{t}{t'}\neq 0$ (see \cite[Theorem 7.1.1]{Chris}). Without loss of generality we can suppose that $t=t_0$. There exists $p\in \Z$ such that $t'=T^pt_0=t_{p}$. By induction we prove that $t_{rp}\in \mathcal{F}_{k_0}$ for all $r\in \Z$. Indeed, the statement holds for $r=0,1$ and since, for $r\geq 1$, $\pin{t_{rp}}{t_{(r+1)p}}=\pin{t_0}{t_{p}}\neq 0$, if $t_{rp}\in \mathcal{F}_{k_0}$ then $t_{(r+1)p}\in \mathcal{F}_{k_0}$ necessarily. In the same way we prove for the statement for $r<0$. \\
	Taking into account that for $j=1,\dots,p-1$, we have  $\pin{f_{rp+j}}{f_{(r+1)p+j}}\neq0$, similarly, there exist $k_1,\dots,$ $k_{p-1}\in \Z$ such that $\{f_{rp+j}\}_{r\in \Z}\subseteq \mathcal{F}_{k_j}$, i.e. $\{f_n\}_{n\in \Z}\subseteq \cup_{j=1}^r \mathcal{F}_{k_j}$. But this leads to a contradiction with (i). 
\end{proof}

\begin{cor}
	\label{cor_Gabor_1}
	Let $\G(g,a,b)$ be an overcomplete Gabor frame with $supp(g)=[c,c+a]$, $c\in \R$. 
	Then, regardless to the ordering over $\Z$, $\G(g,a,b)$ is not generated by an operator in $GL(L^2(\R))$. 
\end{cor}
\begin{proof}
	The canonical dual frame of $\G(g,a,b)$ is a Gabor frame $\G(h,a,b)$ with some $h\in L^2(\R)$ and supp$({h})=[c,c+a]$, because the frame operator $S$ of $\G(g,a,b)$ is a multiplication operator (see \cite[Theorem 6.4.1]{Groc}). The decomposition $\G(g,a,b)=\cup_{n\in \Z} \G_n$, where $\G_n=\{g_{m,n}: m\in \Z\}$, satisfies the hypotheses of Theorem \ref{th_main1} and thus we obtain the conclusion. 
\end{proof}

\begin{rem}
	\label{rem_union_orth_2}
	Corollary \ref{cor_Gabor_1} also shows that in a Hilbert space $\H$ and for any integer $N>1$ there exist $N$ orthonormal bases $\mathcal{F}_1, \mathcal{F}_2, \dots, \mathcal{F}_N$ such that $\cup_{i=1}^N\mathcal{F}_i$ is linearly independent but is not an orbit of an operator in $GL(\H)$ for any ordering of $\cup_{i=1}^N\mathcal{F}_i$.	 
\end{rem}

\begin{rem} 
	Another type of structured frames is constituted by the {\it wavelet frames}. To introduce them we need the {\it dilation operator} $D_c$ ($c>0$) defined as $(D_c f)(x)=c^{-\mez}f(\frac{x}{c})$ which is unitary, like the operators $T_a$ and $E_b$. A frame of $L^2(\R)$ is a wavelet frame if it has the form of $\{D_c^jT_a^k\phi\}_{j,k\in \Z}$ for some $\phi\in L^2(\R)$. Moreover, a tight wavelet frame is linearly independent. \\ 
	An example of overcomplete tight wavelet frame is $\{D_2^jT_1^k\psi\}_{j,k\in \Z}$ where $\psi$ has Fourier transform $\hat{\psi}=\chi_{K}$, the characteristic function on  $K=[-\mez,-\frac{1}{4}]\cup[\frac{1}{4},\mez]$ (see \cite[Example 16.3.6]{Chris}). 
	Here we show that $\{\psi_{j,k}\}:=\{D_2^jT_1^k\psi\}_{j,k\in \Z}$ is not a orbit of an operator in $GL(L^2(\R))$ for any indexing over $\Z$ considered. It is simpler to work with the Fourier transform  $\widehat{\psi_{j,k}}=-D_2^{-j}E_1^k\chi_K$ that have support on $[-\frac{1}{2^{j+1}}, -\frac{1}{2^{j+2}}]\cup[\frac{1}{2^{j+2}},\frac{1}{2^{j+1}}]$. We have $\pin{\psi_{j,k}}{\psi_{m,n}}=\pin{\widehat{\psi_{j,k}}}{\widehat{\psi_{m,n}}}=0$ if $j\neq m$, so Theorem \ref{th_main1} with the decomposition $\{\psi_{j,k}\}_{j,k\in \Z}:=\cup_{j\in \Z}\{\psi_{j,k}\}_{k\in \Z}$ makes the conclusion.
\end{rem}

\begin{rem}
	In the statement of Theorem \ref{th_main1} we cannot replace the infinite union with a finite union and obtain the same result. In fact, let us consider the Hilbert space $L^2(\mathbb{T})$, where $\mathbb{T}:=\{z\in \C: |z|=1\}$. For $z\in \mathbb{T}$ we indicate by $\theta$ the unique number in $[0,1)$ such that $z=e^{2\pi i \theta}$. Fix $K\in \N$, $b=\frac{1}{N}$ with $N\in \N$. 
	Define for $m\in \Z$ and $0\leq k\leq K-1$
	$$f_{m,k}(z)=\begin{cases}
	z^{mb} \quad\text{ if $\frac{k}{K}\leq \theta < \frac{k+1}{K}$ }\\
	0\quad\quad \;\;\text{otherwise.}
	\end{cases}
	$$
	The sequence $\cup_{k=0}^{K-1}\{f_{m,k}\}_{m\in \Z}$ is a linearly independent overcomplete tight frame and the sets $\{f_{m,k}\}_{m\in \Z}$, $\{f_{m,k'}\}_{m\in \Z}$ are orthogonal if $k\neq k'$. Moreover, order the frame over $n\in \Z$ as $f_n=f_{m,k}$ where $m\in \Z$ and $0\leq k<K$ are the quotient and the remainder of $n$ divided by $K$, respectively. We have $\{f_n\}_{n\in \Z}=\{T^n f_0\}_{n\in \Z}$, 	where $T$ is the operator on $GL(L^2(\mathbb{T}))$ defined on $f\in L^2(\mathbb{T})$ as
	$$
	(Tf)(z)=\begin{cases}
	z^b f(e^{2\pi i \frac{1}{K}}z)\quad\text{ if $\frac{k}{K}\leq \theta < \frac{k+1}{K}$ }\\
	f(e^{2\pi i \frac{1}{K}}z)\quad\quad \;\text{otherwise.}
	\end{cases}
	$$
\end{rem}

The following last result includes the second class of Gabor frames that we provide as counterexample of our question. Note that there exist $g\in L^2(\R)$ generating Gabor frames for all $a,b>0$ such that $ab<1$ (or even for all $a,b>0$ such that $ab\leq1$), see \cite[Section 11.6]{Chris}.  
As for the first class (Corollary \ref{cor_Gabor_1}) the proof indicates that the Gabor frames contain to many elements, in an appropriate sense.

\begin{theo}
	\label{th_main2}
	Let $\mathcal{G}(g,a,b)$ be a Gabor frame such that $ab\in \R\backslash\Q$. Then $\mathcal{G}(g,a,b)$, with any ordering over $\Z$, is not generated by an operator in $GL(L^2(\R))$. 
\end{theo}
\begin{proof}
	Fix an ordering $\mathcal{G}(g,a,b)=\{f_n\}_{n\in \Z}$ and suppose that there exists $T\in GL(L^2(\R))$ such that $\{f_n\}_{n\in \Z}=\{T^nf_0\}_{n\in \Z}$. Let $B$ be an upper bound for the canonical dual $\mathcal{G}(h,a,b)$ of $\mathcal{G}(g,a,b)$. 
We distinguish two cases and prove a contradiction in both situations. 
	\begin{itemize}
		\item[(I)] For every $(m,n)\in \Z\times \Z$ with $m\neq0$, one has $\pin{g}{h_{m,n}}=0$. Acting with the operators $E_b$ and $T_a$ we have also $\pin{g_{r,s}}{h_{m,n}}=0$ for every $(m,n)\in \Z\times \Z$ with $m\neq r$. Therefore we have $\mathcal{G}(g,a,b)=\cup_{m\in \Z} \mathcal{G}_m$, $\mathcal{G}(h,a,b)=\cup_{m\in \Z} \mathcal{J}_m$ where $\mathcal{G}_m=\{g_{m,n}\}_{n\in \Z}$, $\mathcal{J}_m=\{h_{m,n}\}_{n\in \Z}$ and $\mathcal{G}_m \perp \mathcal{J}_{m'}$ if $m\neq m'$. We are now in the case of Theorem \ref{th_main1}, thus $\mathcal{G}(g,a,b)$ is not generated by an operator in $GL(L^2(\R))$. 
		\item[(II)] There exists  $(m,n)\in \Z\times \Z$ with $m\neq 0$ such that $\pin{g}{h_{m,n}}\neq 0$. By the commutation rule, for any $k\in \Z$,   
		$$
		\pin{g_{0,k}}{h_{m,n+k}}=e^{2\pi i mkab} \pin{g}{h_{m,n}}\neq 0.
		$$ 
		Since $mkab$ is never an integer for $k\neq 0$,  $\pin{g_{0,k}}{h_{m,n+k}}$ is a different complex number varying $k\in \Z$. By Corollary \ref{cor_nec}, $\{\pin{g_{0,k}}{h_{m,n+k}}\}_{k\in \Z}\subseteq Ip(g)$ and this leads to the contradiction  that
		\begin{align*}
		B\n{g}^2&\geq \sum_{n,m\in \Z} |\pin{g}{h_{m,n}}|^2\geq \sum_{x\in Ip(g)} |x|^2 \\
		&\geq \sum_{k\in \Z} |e^{-2\pi i mkab} \pin{g}{h_{m,n}}|^2=\infty.
		\qedhere
		\end{align*}
	\end{itemize} 
\end{proof}

\section{Conclusions}

In this paper we gave a negative answer to the problem of orbit representation of overcomplete Gabor frames in two cases. The problem is still open for an overcomplete Gabor frame $\G(g,a,b)$ with rational density $ab$ and with window $g$ whose support is not contained in an interval of length $a$. 

We can actually simplify the question. Indeed, acting with the dilation $D_{a^{-1}}$ we can transform $\G(g,a,b)$ in $\G(D_{a^{-1}} g,1,ab)$ and if $\G(g,a,b)$ can be ordered and written as $\{T^n\gamma\}_{n\in \Z}$ with $T\in GL(L^2(\R))$, then $\G(D_{a^{-1}} g,1,ab)$ can be ordered and written as $\{V^nD_{a^{-1}}\gamma\}_{n\in \Z}$ where $V=D_{a^{-1}}TD_{a}\in GL(L^2(\R))$. Thus the problem about orbit representation can be confined to Gabor frames $\G(g,1,b)$ with $b\in \Q$ (or, equivalently, $\G(g,a,1)$ with $a\in \Q$).

\section*{Acknowledgments}

This work was partially supported by the ``Gruppo Nazionale per l'Analisi Matematica, la Probabilità e le loro Applicazioni'' (GNAMPA – INdAM).

\vspace*{0.4cm}
\begin{center}
\textsc{Rosario Corso, Dipartimento di Matematica e Informatica} \\
\textsc{Università degli Studi di Palermo, I-90123 Palermo, Italy} \\
{\it E-mail address}: {\bf rosario.corso02@unipa.it}
\end{center}

\end{document}